\documentclass[12pt,reqno]{amsart}
\usepackage{amsmath,amssymb,amsfonts,amscd,latexsym,amsthm,mathrsfs,verbatim,cyrillic}
\usepackage[unicode]{hyperref}
\usepackage{hypbmsec}
\newcommand{\cyrrm}{\fontencoding{OT2}\selectfont\textcyrup}
\newtheorem{theorem}{Theorem}

\newtheorem{prop}{Proposition}
\theoremstyle{remark}

\let\eps\varepsilon

\renewcommand{\d}{{\mathrm d}}

\newcommand{\lcm}{\operatorname{lcm}}

\newcommand{\bx}{{\boldsymbol x}}
\newcommand{\bz}{{\boldsymbol z}}
\begin{document}

\title{A determinantal approach to irrationality}

\author{Wadim Zudilin}
\address{School of Mathematical and Physical Sciences, The University of Newcastle, Callaghan NSW 2308, AUSTRALIA}
\email{wzudilin@gmail.com}

\date{21 July 2015. \emph{Revised}: 2 February 2015}

\begin{abstract}
It is a classical fact that the irrationality of a number $\xi\in\mathbb R$ follows from the existence of a sequence $p_n/q_n$ with
integral $p_n$ and $q_n$ such that $q_n\xi-p_n\ne0$ for all $n$ and $q_n\xi-p_n\to0$ as $n\to\infty$. In this note we give an extension
of this criterion in the case when the sequence possesses an additional structure; in particular, the requirement $q_n\xi-p_n\to0$ is weakened.
Some applications are given including a new proof of the irrationality of $\pi$.
Finally, we discuss analytical obstructions to extend the new irrationality criterion further
and speculate about some mathematical constants whose irrationality is still to be established.
\end{abstract}


\thanks{The work is supported by Australian Research Council grant DP170100466.}

\maketitle

\section{Irrationality criteria}
\label{s:irr}

Let $\xi$ be a real number we wish to prove the irrationality of.
Assume there exists a sequence $p_n/q_n$ with integral $p_n$ and $q_n$ such that $\xi\ne p_n/q_n$ for all $n$ and $q_n\xi-p_n\to0$ as $n\to\infty$.
If $\xi$ were rational, $\xi=p/q$ say, then $q(q_n\xi-p_n)=q_np-p_nq$ is a nonzero integer for all $n$, hence its absolute value is at least~1.
On the other hand, $|q(q_n\xi-p_n)|$ is less than~1 for $n$ sufficiently large; a contradiction.

Let us put the above classical irrationality criterion in a standard context of real numbers $\xi$ that happen to be periods.
Suppose we have a sequence of rational approximations
\begin{equation}
r_n=a_n\xi-b_n\in\mathbb Q\xi+\mathbb Q
\label{eq:r_n}
\end{equation}
such that
\begin{itemize}
\item[(a)] $0<r_n\le C_1\eps^n$ for some $C_1,\eps>0$ and all $n=1,2,\dots$;
\item[(b)] $\delta_na_n,\delta_nb_n\in\mathbb Z$ for some sequence of positive integers $\delta_n$; and
\item[(c)] $\delta_n<C_2\Delta^n$ for some $C_2,\Delta>0$ and all $n=1,2,\dots$\,.
\end{itemize}

\begin{prop}
\label{pr1}
Under hypotheses \textup{(a)--(c)}, if $\eps\Delta<1$ then $\xi$ is irrational.
\end{prop}

The principal goal of this note is to demonstrate that under some further (natural) assumptions on the approximants \eqref{eq:r_n}
we can replace the condition $\eps\Delta<1$ by a slightly different one.
Namely, assume additionally that
\begin{itemize}
\item[(d)] $\displaystyle r_n=\int_\gamma z(\bx)^n\omega(\bx)$ for some domain $\gamma\subset\mathbb R^m$,
non-constant continuous function $z(\bx)\ge0$ on $\gamma$ and measure (positive differential form) $\omega(\bx)$; and
\item[(e)] $\delta_n$ divides $\delta_{n+1}$ for all $n=1,2,\dots$\,.
\end{itemize}

\begin{prop}
\label{pr2}
Under hypotheses \textup{(a)--(e)}, if $\eps\Delta^{3/2}/4<1$ then $\xi$ is irrational.
\end{prop}

Our proof, in which the classical Vandermonde determinants as well as Hankel determinants of the sequence \eqref{eq:r_n}
show up, is given in Section~\ref{s:det}.
Note that $\eps\Delta^{3/2}/4<(\eps\Delta)^{3/2}$ if $\eps>1/16$, thus
Proposition~\ref{pr2} has potentials to produce irrationality results when Proposition~\ref{pr1} is not applicable.
For example, if $\Delta\approx e=2.7182\dots$ and $0.34<\eps<0.89$ then Proposition~\ref{pr2} implies the irrationality of $\xi$, while Proposition~\ref{pr1} does not.
In Sections~\ref{s:log3} and \ref{s:pi} we give such applications of the new irrationality criterion to $\log3$ and $\pi$.
Though these numbers are known to be irrational, our proofs based on the argument of Proposition~\ref{pr2} are new.

The example $\xi=1$, $a_n=1$, $b_n=(2^n-1)/2^n$, so that $a_n\xi-b_n=1/2^n$, demonstrates the importance of~(d):
in this case we can take $\eps=1/2$ (or slightly bigger) and $\delta_n=2^n$, hence $\Delta=2$.
The condition $\eps\Delta^{3/2}/4=\sqrt2/4<1$ is clearly satisfied but without any implication,
as there is no way to write the linear forms $a_n\xi-b_n$ in the form assumed in~(d)
for a non-constant function $z(\bx)$. (In fact, all the Hankel determinants that appear in the proof of Proposition~\ref{pr2} below vanish in this case.)
On the other hand, the example $z(x)=x$, $\omega(x)=\d x$ and $\gamma=[0,1]$ (this corresponds to the choice $\xi=1$ or~$0$), so that $\eps=1$,
$$
\delta_n=\lcm(1,2,\dots,n+1)
$$
and $\Delta$ is anything slightly larger than~$e$, shows that the condition $\eps\Delta^{3/2}/4<1$ cannot be relaxed ``too much.''
In Section~\ref{s:diam} we discuss the optimality of the latter constraint as well as comment on the orthogonality induced by the data from condition~(d).
Finally, in Section~\ref{s:const} we speculate about some mathematical constants whose irrationality is still to be established.

\section{Proof of the irrationality criterion}
\label{s:det}

\begin{proof}[Proof of Proposition~\textup{\ref{pr2}}]
Conditions (a) and (d) imply that
$$
0<\sup_{\bx\in\gamma}z(\bx)\le\eps.
$$
Consider the polynomial forms
\begin{equation}
R_n=\frac1{n!}\idotsint_{\gamma^n}\prod_{1\le j<\ell\le n}(z(\bx^\ell)-z(\bx^j))^2\omega^n
\in\mathbb Q\xi^n+\dots+\mathbb Q\xi+\mathbb Q.
\label{eq:R_n}
\end{equation}
Note that $R_n>0$, because the integrand is nonnegative.
Furthermore, the form \eqref{eq:R_n} has a nice Hankel determinant evaluation due to Heine \cite{He78},
\begin{equation}
R_n=\det_{0\le j,\ell<n}(r_{j+\ell}),
\label{eq:Heine}
\end{equation}
which together with hypotheses (b) and (e) imply that
\begin{equation*}
\delta_{n-1}\delta_n\dotsb\delta_{2n-2}R_n\in\mathbb Z\xi^n+\dots+\mathbb Z\xi+\mathbb Z.
\end{equation*}
In addition, we have
\begin{equation}
\biggl(\sup_{\substack{0\le z_j\le\eps\\j=1,\dots,n}}\prod_{1\le j<\ell\le n}(z_j-z_\ell)^2\biggr)^{1/n^2}
\to\eps/4 \qquad\text{as}\quad n\to\infty
\label{eq:Cheb}
\end{equation}
(the Fekete--Chebyshev constant of the interval $[0,\eps]$), so that
\begin{equation*}
R_n<(\eps/4)^{n^2+o(n^2)} \qquad\text{as}\quad n\to\infty.
\end{equation*}

If $\xi$ were rational, $\xi=p/q$, then $\prod_{j=0}^{n-1}\delta_{n-1+j}\cdot q^nR_n$ would be a positive integer for any $n$.
On the other hand,
$$
\delta_{n-1}\delta_n\dotsb\delta_{2n-2}q^nR_n<(C_2q)^n\Delta^{(3/2)n^2}(\eps/4)^{(1+o(1))\,n^2}\to0 \qquad\text{as}\quad n\to\infty,
$$
a contradiction.
\end{proof}

In the argument above we could have used the inclusion $\prod_{j=0}^{n-1}\delta_{n-1+j}\cdot q^n\cdot n!R_n\in\mathbb Z$ instead
(with the factorial factor) as an immediate consequence of the integral representation \eqref{eq:R_n}, no reference to~\eqref{eq:Heine} is necessary.
But because Heine's theorem about the determinant representation of $R_n$ is itself quite elementary and nice, we reproduce its proof here for completeness.

\begin{proof}[Proof of identity~\textup{\eqref{eq:Heine}}]
Using the definition of the sequence $r_n$ write
\begin{align*}
\det_{0\le j,\ell<n}(r_{j+\ell})
&=\det_{0\le j,\ell<n}\biggl(\int_\gamma z(\bx^j)^{j+\ell}\omega(\bx^j)\biggr)
\\
&=\idotsint_{\gamma^n}\det_{0\le j,\ell<n}\bigl(z(\bx^j)^{j+\ell}\bigr)\,\omega(\bx^0)\,\omega(\bx^1)\dotsb\omega(\bx^{n-1})
\\
&=\idotsint_{\gamma^n}\det_{0\le j,\ell<n}(z_j^{j+\ell})\,\omega^n
\\
&=\idotsint_{\gamma^n}z_1z_2^2\dotsb z_{n-1}^{n-1}\prod_{0\le j<\ell\le n-1}(z_\ell-z_j)\,\omega^n
\end{align*}
where $z_j=z(\bx^j)$ and the evaluation
\begin{equation}
\det_{0\le j,\ell<n}(z_j^\ell)=\prod_{0\le j<\ell\le n-1}(z_\ell-z_j)
\label{eq:Vand}
\end{equation}
is applied. It remains to notice that the latter Vandermonde determinant is invariant, up to multiplication by $\operatorname{sgn}(\sigma)$,
under the transformations $\sigma\in\mathfrak S_n$ of the $n$-element set $\{0,1,\dots,n-1\}$ of indices and that
$$
\sum_{\sigma\in\mathfrak S_n}\operatorname{sgn}(\sigma)\,z_{\sigma(1)}z_{\sigma(2)}^2\dotsb z_{\sigma(n-1)}^{n-1}
=\det_{0\le j,\ell<n}(z_j^\ell)
$$
is the same determinant~\eqref{eq:Vand}.
\end{proof}

\section{Some applications of the criteria}
\label{s:log3}

For a (real or complex) $a\ne0,1$, introduce the sequence of quantities
\begin{equation}
I_n=I_n(a)=\int_1^a\frac{(x-1)^n(a-x)^n}{x^{n+1}}\,\d x,
\qquad n=0,1,2,\dots\,.
\label{eq:1}
\end{equation}
Applying the binomial theorem to each of the factors in the numerator of integrand and then integrating we find out that
\begin{align*}
I_n
&=\sum_{j,\ell=0}^n\binom nj\binom n\ell(-1)^{n+j+\ell}a^{n-\ell}\int_1^ax^{j+\ell-n-1}\d x
\\
&=\sum_{\substack{j,\ell=0\\j+\ell\ne n}}^n\binom nj\binom n\ell(-1)^{n+j+\ell}\,\frac{a^j-a^{n-\ell}}{j+\ell-n}
+(\log a)\sum_{j=0}^n{\binom nj}^2a^j,
\end{align*}
where $\log a$ is understood as the integral of $(\d x)/x$ along a given path from 1 to $a$ in~\eqref{eq:1}.
Since $|j+\ell-n|\le n$ in each summand of the first sum, we conclude from the representation obtained that
\begin{equation}
d_nI_n\in\mathbb Z[a]\,\log a+\mathbb Z[a],
\qquad n=0,1,2,\dots,
\label{eq:1a}
\end{equation}
where $d_n$ denotes the least common multiple of the integers from $1$ to~$n$ (and $d_0=1$).
Note that $d_n^{1/n}\to e$ as $n\to\infty$ by the prime number theorem.
If $a>1$ is an integer then the inclusions \eqref{eq:1a} simply mean that $d_nI_n\in\mathbb Z\,\log a+\mathbb Z$ for $n=0,1,2,\dots$\,.

The family above corresponds to the choice
$$
z(x)=\frac{(x-1)(a-x)}x, \quad \omega(x)=\frac{\d x}x \quad\text{and}\quad \gamma=[1,a]\subset\mathbb R,
$$
in the notation of (a)--(e). When $a>1$, one easily finds that
$$
\max_{1\le x\le a}\{z(x)\}=z(\sqrt a)=(\sqrt a-1)^2.
$$

\begin{theorem}
\label{pr3}
$\log2$ and $\log3$ are irrational.
\end{theorem}

\begin{proof}
As $(\sqrt2-1)^2e=0.4663\ldots<1$, the irrationality of $\log2$ follows already from application of Proposition~\ref{pr1}.
In the case of $\log3$ we use $(\sqrt3-1)^2e^{3/2}/4=0.6004\ldots<1$ and Proposition~\ref{pr2}.
\end{proof}

\section({A new proof of the irrationality of \003\300}){A new proof of the irrationality of $\pi$}
\label{s:pi}

We can also use the above argument for $a=i=\sqrt{-1}$, when the integrals in~\eqref{eq:1} produce approximations
to $\pi/2=-i\,\log i$ with coefficients from $\mathbb Q[i]$. (For all practical purposes we can think of integration in~\eqref{eq:1}
as going along the arc of the unit circle.)
In fact, the change of variable $x=(1+it)/(1-it)$ transforms the integrals into
\begin{equation*}
I_n=I_n(i)=2^{n+1}i(-1-i)^n\int_0^1\frac{t^n(1-t)^n}{(1+t^2)^{n+1}}\,\d t,
\qquad n=0,1,2,\dots;
\end{equation*}
therefore, it follows from \eqref{eq:1a} and the latter that
\begin{equation}
2i(1-i)^{4\{n/4\}}d_nI_n\in\mathbb Z\,\pi+\mathbb Z,
\qquad n=0,1,2,\dots,
\label{eq:1b}
\end{equation}
in this case, where $\{\,\cdot\,\}$ denotes the fractional part.

\begin{theorem}
\label{pr4}
$\pi$ is irrational.
\end{theorem}

\begin{proof}
We use
$$
\max_{x\in\operatorname{arc}(1,i)}\biggl|\frac{(x-1)(i-x)}x\biggr|
=2^{3/2}\max_{t\in[0,1]}\frac{t(1-t)}{1+t^2}=2-\sqrt2=0.5857\ldots
$$
and the inclusions~\eqref{eq:1b}. Since $(2-\sqrt2)e^{3/2}/4=0.6563\ldots<1$, (a slight adaptation of) Proposition~\ref{pr2}
implies that the approximated number $\pi$ is irrational.
\end{proof}

We remark that the potentials of the integral construction
$$
\int_0^1\frac{t^n(1-t)^n}{(1+t^2)^{n+1}}\,\d t
\in\mathbb Q\pi+\mathbb Q
$$
from the classical perspective (that is, Proposition~\ref{pr1})
are already discussed in the section ``A second attempt'' in \cite[p.~375]{Be00}:
The first few approximations to $\pi$ produced by the integrals look promising,
``[u]nfortunately in the long run the asymptotics have decided otherwise.''
Proposition~\ref{pr2} can be applied (though not directly since the corresponding $r_n$ there have
to be replaced with $\sqrt2\,r_n$ for $n$~odd to accommodate the arithmetic part of the argument)
to the integrals by choosing
$$
z(t)=\frac{2^{3/2}t(1-t)}{1+t^2}, \quad \omega(t)=\frac{\d t}{1+t^2}, \quad \gamma=[0,1]
$$
and $\delta_n=4d_n$.
Note the related elementary identity
\begin{equation}
\det_{0\le j,\ell<n}(c^{j+\ell}v_{j+\ell})
=c^{n(n-1)}\det_{0\le j,\ell<n}(v_{j+\ell}),
\qquad c\in\mathbb C\setminus\{0\},
\label{eq:red}
\end{equation}
that allows to easily manipulate with the extra factors like $2^{3/2}$ above in computing the Hankel-determinant asymptotics.

\section{Commentary}
\label{s:diam}

Though Proposition~\ref{pr2} does not sound very practical, a question is about how much we can relax hypothesis (d)
to still possess its implication.

One possibility is to ``densify'' the sequence of linear forms $r_n$ and consider instead the sequence
$\hat r_n=r_{\lfloor n/k\rfloor}$ for some fixed integral $k\ge2$ and together with the corresponding Hankel
determinants $\hat R_n=\det_{0\le i,j<n}(\hat r_{i+j})$. Surprisingly enough (as not discovered in the existing literature)
the underlying ``$k$-stuttering'' Vandermonde determinants
$$
\det_{0\le j,\ell<n}(z_j^{\lfloor(j+\ell)/k\rfloor})
$$
(that replace the classical Vandermonde determinants \eqref{eq:Vand} in the above proof of Heine's identity \eqref{eq:Heine})
factors into the product of powers of $z_j$ and of $k$ Vandermonde determinants; roughly speaking, $\hat R_n$ behaves similar to $R_{\lfloor n/k\rfloor}^k$,
thus leading to no refinement of Proposition~\ref{pr2}.
Furthermore, we recall that the non-vanishing of $\hat R_n$ for infinitely many indices $n$ can be shown without the reference
to a generalization of Heine's identity but with the use of the equality
$$
\sum_{n=0}^\infty\hat r_nz^n
=(1+z+\dots+z^{k-1})\sum_{n=0}^\infty r_nz^{kn}
$$
and of an old result of Kronecker (see \cite[pp.~566--567]{Kr81} or
\cite[Division~7, Problem~24]{PS76}): the Hankel determinants $\det_{0\le j,\ell\le n-1}(v_{i+j})$ for $n=0,1,\dots$
eventually vanish if and only if the generating series $\sum_{n=0}^\infty v_nz^n$ represents a rational function.

Further variations are still possible, for example, considering other even powers of the Vandermonde determinants in~\eqref{eq:R_n}\,---\,this
has a nice interpretation by means of the Hankel hyperdeterminants (see \cite{LT03}) or replacing the Vandermonde determinants
by other polynomials $p(\bz)\in\mathbb Z[z_1,\dots,z_n]$ (and ensuring that the corresponding integrals over $\gamma^n$ do not vanish).
But there is a natural analytical obstruction to getting anything weaker than $\eps\Delta^{3/2}/4<1$. It is based on the fact that
the $n$-variate $\mathbb Z$-transfinite diameter $t_{\mathbb Z}([0,\eps]^n)$, which is introduced and studied in \cite{BP09}
for general sets $E\subset\mathbb C^n$, tends to $\eps/4$ as $n\to\infty$\,---\,see Proposition~\ref{PR3} below.
This means that for a nonzero $n$-variable polynomial $p(\bz)$ of total degree $N$ with integral coefficients we have
$$
\sup_{\bz\in[0,\eps]^n}|p(\bz)|\ge\rho_n^N,
$$
where $\rho_n\to\eps/4$ as $n\to\infty$,
thus showing that the upper estimate \eqref{eq:Cheb} is best possible,
as the total degree of the polynomial $\prod_{1\le j<\ell\le n}(z_j-z_\ell)^2$ is $n^2-n$.

\begin{prop}
\label{PR3}
For the multivariate $\mathbb Z$-transfinite diameter introduced in \textup{\cite{BP09}} we have
$$
\lim_{n\to\infty}t_{\mathbb Z}([a,b]^n)=\frac{|b-a|}4.
$$
\end{prop}

\begin{proof}
There are two $\mathbb C$-extensions of the $n$-variate $\mathbb Z$-transfinite diameter, called $\tau(E)$ and $T(E)$ in \cite{Za75}.
The Hilbert-type relation between $t_{\mathbb Z}(E)$ and $\tau(E)=t_{\mathbb C}(E)$ is established in \cite[Theorem 3.1]{BP09}:
\begin{equation}
t_{\mathbb Z}(E)\le\tau(E)^{n/(n+1)}.
\label{e1}
\end{equation}
The estimate
\begin{equation}
T(E)\le t_{\mathbb Z}(E)
\label{e2}
\end{equation}
trivially follows from the definition of the two characteristics. Furthermore, the bound $T(E)\ge c(E)$, where $c(E)$ is the capacity of $E$ is proved in~\cite{Za75}.
If we restrict our attention to the cartesian product $E=[a,b]^n$ then $c(E)=c([a,b])$ by \cite[property~{\cyrrm{d}}$'$)]{Za75}
and $\tau(E)=\tau([a,b])$ by \cite[property~{\cyrrm{d}})]{Za75}. It remains to use $\tau([a,b])=c([a,b])=|b-a|/4$ to conclude from \eqref{e1} and \eqref{e2} that
the required limit relation holds true.
\end{proof}

Our further remark refers to the fact that the data $\gamma\subset\mathbb R^m$, $z(\bx)$ and $\omega(\bx)$ from condition~(d) in Section~\ref{s:irr}
give rise to the scalar product on a space of single variable functions of $z\in\mathbb R$:
$$
\langle u,v\rangle=\int_\gamma u(z(\bx))v(z(\bx))\,\omega(\bx),
$$
so that the sequence $r_n$ in (d) is nothing but the sequence of the corresponding moments.
It may be of some interest to study the orthogonal polynomials $p_n(z)$ for $n=0,1,\dots$ arising from the product:
the polynomial $p_n(z)$ can be explicitly written as the determinant $R_{n+1}$ in \eqref{eq:Heine}, in which the last row
is replaced with the row $1,z,\dots,z^n$, or as the integral
$$
p_n(z)=\frac1{n!}\idotsint_{\gamma^n}\prod_{j=1}^n(z-z(\bx^j))\prod_{1\le j<\ell\le n}(z(\bx^\ell)-z(\bx^j))^2\omega(\bx^1)\dotsb\omega(\bx^n).
$$
Plenty of the theory of orthogonal polynomials naturally extends to these general settings,
though it is not clear how much of this can be used in the irrationality context.

\section{Catalan and $q$-Ap\'ery constants}
\label{s:const}

It is worth mentioning that Proposition~\ref{pr2} is applicable to the integrals
$$
d_n^2\iint_{[0,1]^2}\biggl(\frac{x(1-x)y(1-y)}{1-xy}\biggr)^n\frac{\d x\,\d y}{1-xy}
\in\mathbb Z\,\zeta(2)+\mathbb Z
$$
and
$$
d_n^3\iiint_{[0,1]^3}\biggl(\frac{x(1-x)y(1-y)z(1-z)}{1-(1-xy)z}\biggr)^n\frac{\d x\,\d y\,\d z}{1-(1-xy)z}
\in\mathbb Z\,\zeta(3)+\mathbb Z
$$
used by Beukers in his proof \cite{Be79} of Ap\'ery's theorem about the irrationality of $\zeta(2)$ and of $\zeta(3)$, the Ap\'ery constant.
This follows from
\begin{gather*}
\max_{(x,y)\in[0,1]^2}\frac{x(1-x)y(1-y)}{1-xy}=\biggl(\frac{\sqrt5-1}2\biggr)^5,
\\
\max_{(x,y,z)\in[0,1]^3}\frac{x(1-x)y(1-y)z(1-z)}{1-(1-xy)z}=(\sqrt2-1)^4,
\end{gather*}
the prime number theorem and calculation
$$
\frac14\biggl(\frac{\sqrt5-1}2\biggr)^5e^3=0.4527\ldots<1,
\qquad
\frac14(\sqrt2-1)^4e^{9/2}=0.6624\ldots<1.
$$
For the related construction of rational approximations
$$
2^{4n+1}d_{2n}^2\iint_{[0,1]^2}\biggl(\frac{x(1-x)y(1-y)}{1-xy}\biggr)^n\frac{x^{-1/2}(1-y)^{-1/2}\,\d x\,\d y}{1-xy}
\in\mathbb Z\,G+\mathbb Z
$$
to the Catalan constant
$$
G=\sum_{k=0}^\infty\frac{(-1)^k}{(2k+1)^2}
$$
(see \cite{Ne15,Ri06,Zu02}), none of the irrationality criteria works, because of the too impetuous growth of the denominators required.

Other ways of ``densifying'' the sequence of rational approximations, different from the one discussed in Section~\ref{s:diam},
can be proposed, that take into account the expression of $z(\bx)$. For example, in the case of the Catalan constant we
can choose rational approximations to be
\begin{align*}
\tilde r_n
&=\iint_{[0,1]^2}
\frac{x^{\lfloor(n+1)/5\rfloor}(1-x)^{\lfloor(n+2)/5\rfloor}y^{\lfloor(n+3)/5\rfloor}(1-y)^{\lfloor(n+4)/5\rfloor}}{(1-xy)^{\lfloor n/5\rfloor}}
\\ &\qquad\times 
\frac{x^{-1/2}(1-y)^{-1/2}\,\d x\,\d y}{1-xy}
\in\mathbb Q\,G+\mathbb Q,
\qquad n=0,1,2,\dots,
\end{align*}
because the corresponding Hankel determinants \emph{without} the integer parts in the exponents behave asymptotically like $(\eps^{1/5}/4)^{n^2}$,
where $\eps=((\sqrt5-1)/2)^5$. However, numerical computation of the honest Hankel determinants $\tilde R_n=\det_{0\le j,\ell<n}(\tilde r_{j+\ell})$ suggests
$\tilde R_n^{1/n^2}\to(\eps/4)^{1/5}$ as $n\to\infty$;
this numerical observation can be rigorously justified using again the ``5-stuttering'' Vandermonde determinants.

Another interesting question is about possible $q$-extensions of Proposition~\ref{pr2};
the work \cite{KRVZ09} suggests that there may be some.
For example, in \cite{KRZ06} $q$-analogues of the Ap\'ery--Beukers rational approximations to $\zeta(3)$ were constructed,
which approximate the $q$-series
$$
\zeta_q(3)=\sum_{k=1}^\infty\sigma_2(k)q^k=\sum_{m=1}^\infty\frac{m^2q^m}{1-q^m}=\sum_{k=1}^\infty\frac{q^k(1+q^k)}{(1-q^k)^3},
$$
for $q$ the reciprocal of an integer different from $0$ and $\pm1$. It is expected that this $q$-Ap\'ery constant is irrational for such values of~$q$,
but the corresponding $q$-approx\-imations $r_n$ from~\cite{KRZ06} do not produce any irrationality result,
because $|r_n|^{1/n^2}\sim1$ while $\delta_n^{1/n^2}\sim|q|^{-9/\pi^2}$ as $n\to\infty$, where $\delta_nr_n\in\mathbb Z\zeta_q(3)+\mathbb Z$.
Though no arithmetic consequences come out from considering the Hankel determinants $R_n=\det_{0\le j,\ell<n}(r_{j+\ell})$,
an analytic argument \cite[Section 4]{Zu16} shows their (better than expected) behaviour $|R_n|^{1/n^3}\sim|q|^{1/3}$ as $n\to\infty$.

More ideas are required to deal with the quantities like $G$ and $\zeta_q(3)$.

\section*{Acknowledgements}

There are several inspirations for this project, the most recent one being the work \cite{Br14} of Francis Brown
and, in particular, his remark there: ``Much more optimistically still, one might hope to prove
the transcendence of $\zeta(3)$ by optimizing our polynomial forms in $\zeta(3)$ along the lines of~\cite{So96}.''
The other sources of inspirations include my joint work \cite{KRVZ09} with Christian Krattenthaler, Igor Rochev and Keijo V\"a\"an\"anen
on (related) Hankel-determinant constructions for certain $q$-hypergeometric series, and also the work \cite{Mo09} of Hartmut Monien
on (unrelated) Hankel determinants on the values of Riemann's zeta function at positive integers.
I thank all these colleagues as well as Igor Pritsker for numerous helpful chats about the topic of this project.
Furthermore, I am very grateful to St\'ephane Fischler whose constructive feedback was crucial at several places of the preliminary version.
Special thanks go to the anonymous referee of the journal for his healthy criticism.

Part of the work was done during my visit in the Max Planck Institute for Mathematics, Bonn, in March--April 2015.
I am thankful to the staff and guests of the institute for creating the unique ``mathemagical'' atmosphere for scientific performance.


\end{document}